\documentclass{amsart}

\usepackage{lineno}

\newtheorem{thm}{Theorem}[section]

\newtheorem{lemma}[thm]{Lemma}
\newtheorem{claim}[thm]{Claim}

\theoremstyle{definition}
\newtheorem{defn}[thm]{Definition}

\theoremstyle{remark}
\newtheorem{remark}{Remark}[section]

\begin{document}

\title{A random link via bridge position is hyperbolic}

\author{Kazuhiro Ichihara}
\address{Department of Mathematics, College of Humanities and Sciences, Nihon University, 3-25-40 Sakurajosui, Setagaya-ku, Tokyo 156-8550, Japan}
\email{ichihara@math.chs.nihon-u.ac.jp}

\author{Jiming Ma}
\address{School of Mathematical Sciences, Fudan University, Shanghai, 200433, P. R. China}
\email{majiming@fudan.edu.cn}



\date{\today}

\thanks{Ichihara was partially supported by JSPS KAKENHI Grant Number 26400100. Jiming Ma was partially supported by  NSFC  11371094.}

\begin{abstract} We show that a random link defined by random bridge splitting is hyperbolic with asymptotic probability 1.\end{abstract}


\maketitle

\section{Introduction}

In recent low-dimensional topology, to study 3-manifolds and knot via \textit{random methods} could be a hot topic now. 

As a pioneering work, in \cite{DunfieldThurston:2006}, Dunfield and Thurston introduced a random model of 3-manifolds by using random walks on the mapping class group of a surface, and a theory of random 3-manifolds has started. 
Actually they considered random Heegaard splittings by gluing a pair of handlebodies by the result of a random walk in the mapping class group. 

Later random Heegaard splittings are studied extensively by Maher in \cite{Maher:2010}. 
In particular, he showed that a 3-manifold obtained by a random Heegaard splitting is hyperbolic with asymptotic probability 1.

As a natural generalization, a random bridge decomposition for a link in the 3-sphere $S^3$ was considered and studied by the second author in \cite{Ma:2013}. 
There he computed the expected value of the number of components of a random link, via the random braid model and the random bridge presentation model. 
A further study to this for random braid model is given in \cite{IchiharaYoshida:2015}. 

Moreover, for random braid model, Ma showed in \cite{Ma:2014} that, for a random walk on the braid group, the probability that the link appearing as the braid closure is hyperbolic converges to 1. 

In this paper, as a generalization to the result of Maher, and a dual result of Ma, we show that a link obtained by a random bridge decomposition is hyperbolic with asymptotic probability 1. 

To state our result precisely, we set up our terminology. 
Let $B = B_+$ be a 3-ball with trivial (i.e., boundary-parallel) $n$-strings $\tau_+$ in it, and $B_-$ and $\tau_-$ their mirror images. 
Regarding $(B_*, \tau_*)$ as the trivial tangle, the boundary of $(B_* ,\tau_*)$, meaning that $\partial B_* - \tau_*$, is a $2n$-punctured sphere $S_{0,2n}$ for $*=+$ or $-$. 
Let us denote by $\mathfrak{M}_{0,2n}$ the (orientation-preserving) mapping class group of $S_{0,2n}$. 
Then, for $\varphi  \in \mathfrak{M}_{0,2n}$ with a representative $f : S_{0,2n} \to S_{0,2n}$, 
$(B_+ - \tau_+) \cup_f (B- - \tau_-)$ gives the complement of a link $L$ in the 3-sphere $S^3$ so that the sets of strings $\tau_+$ and $\tau_-$ are glued together to get $L$. 
We denote by $L = \bar{\varphi}$, and call it a link with an $n$-bridge presentation
$(B_+, \tau_+) \cup_\varphi (B_- , \tau_-)$. 

Now the statement of our main result is the following. 

\begin{thm} \label{thm: randomlinkhyp}  
Let $w _{k}$ be a random walk in the mapping class group $\mathfrak{M}_{0,2n}$ of $2n$-punctured sphere for $n \ge 3$. 
Suppose that $w _{k}$ is generated by a finitely supported probability distribution $\mu$ on $\mathfrak{M}_{0,2n}$, 
whose support generates a semi-group containing a complete subgroup. 
Then the probability that the link $\overline{w _{k}}$ is hyperbolic tends to $1$ as $n \to \infty$. 
\end{thm}

Here a subgroup of the mapping class group is called \textit{complete} if  the endpoints of its pseudo-Anosov elements are dense in the space of projective measured foliations on the surface. 
We here omit the details about the mapping class groups on surfaces and the random walks on groups. 
Please refer to \cite[Section 2]{Maher:2010}. 

We remark that we except the special case $n=2$, that is, the 2-bridge link case. 
For technical reason, it should be treated separately, but one could obtain the same result in the same way. 

Our proof is similar to the approach of Maher  \cite{Maher:2010},  but  in the approach, we show two results on the Hilden groups, which are similar to results in the handlebody groups, which should be interesting and deserve to write down for the purpose of knot  theory.

\section{Outline of Proof}

In this section, we give an outline of the proof of Theorem \ref{thm: randomlinkhyp}. 
Our proof just runs along the same line as the proof of \cite[Theorem 1.1]{Maher:2010} by Maher, which presents the same result for closed 3-manifolds. 

As in Theorem \ref{thm: randomlinkhyp}, let $w _{k}$ be a random walk in the mapping class group $\mathfrak{M}_{0,2n}$ of $2n$-punctured sphere for $n \ge 3$. 
We suppose that $w _{k}$ is generated by a finitely supported probability distribution $\mu$ on $\mathfrak{M}_{0,2n}$, whose support generates a semi-group containing a complete subgroup. 

We primarily want to apply the following result due to Maher to our setting. 

\begin{quote}
\cite[Theorem 5.5]{Maher:2010} 
Let $G$ be the mapping class group of a orientable surface. 
Consider a random walk generated by a finitely supported probability distribution $\mu$ on $G$, and 
whose support generates a semi-group containing a complete subgroup of $G$. 
Let $X$ be a quasi-convex subset of the relative space $\hat{G}$, 
whose limit set has measure zero with respect to both harmonic measure and reflected harmonic measure. 
Then there is a constant $\ell > 0$ such that
$$
\mathbb{P} \left(\left| \frac{1}{n} \hat{d} ( X , w_n X ) - \ell \right| \le  \epsilon \right) \to 1 \mbox{ as } n \to \infty ,
$$
for all $\epsilon > 0$, where $\mathbb{P}$ is the probability measure determined by the probability distribution $\mu$ and $\hat{d} ( X , w_n X )$ is the minimum distance between $X$ and $w_n X$.
\end{quote}

\begin{remark}
In \cite{Maher:2010}, this theorem, Theorem 5.5, is originally stated with the assumption that the surface considered is closed. 
However, actually for the proof of Theorem 5.5, we do not need closedness of the surface. 
It works fine for surfaces with boundary as long as X has the required properties.
\end{remark}

Here we recall definitions used above briefly. 

The \textit{relative space} $\hat{G}$ of the mapping class group $G$ is defined as follows. 
Let $\{ \alpha_1, \dots , \alpha_n \}$ be a list of representatives of orbits of simple closed curves on the surface under the action of $G$. 
Let $H_i = \mbox{fix} (\alpha_i)$ be the subgroup of $G$ fixing $\alpha_i$, and $\mathcal{H} = \{ H_i \}$ the collection of the subgroups $H_i$'s. 
With a generating set $A$ of $G$, we define the relative length of an element $g \in G$ to be the length of the shortest word in the infinite generating set $A \cup \mathcal{H}$. 
This defines a metric $\hat{d}$ on $G$ called the \textit{relative metric}, and 
we call the group $G$ with the relative metric $\hat{d}$ the \textit{relative space} denoted by $\hat{G}$. 
Remark that any two different choices of finite generating set give quasi-isometric relative metrics. 

Based on a work of Kaimanovich and Masur \cite{KaimanovichMasur:1996}, it is shown by Maher \cite{Maher:2011} that a random walk on the mapping class group $G$ on a surface $F$ converges to a uniquely ergodic foliation (lamination) in the Gromov boundary of the curve complex of $F$ almost surely. 
Since the subset of the uniquely ergodic foliation (lamination) is regarded as a subset of the space of projective measured foliations $\mathcal{PML}(F)$, we see that a random walk on the mapping class group $G$ converges $\mathcal{PML}(F)$ almost surely. 
This defines a measure $\nu$, known as \textit{harmonic measure}, on $\mathcal{PML}(F)$. 
Actually the measure $\nu$ of a subset is just defined as the probability that a sample path of the random walk converges to a foliation contained in that subset. 
We remark that this measure depends on the choice of probability distribution used to generate the random walk. 
For example, if we consider the reflected random walk generated by $\mu (g) = \mu(g^{-1})$, then if $\mu$ is not symmetric, this may produce a different harmonic measure, which we shall call the \textit{reflected harmonic measure}.

\bigskip

Following the proof of \cite[Theorem 1.1]{Maher:2010} by Maher, we want to consider \textit{the disk set of the curve complex} as the set $X$ above. 
To do so, we prepare a lemma, which is independent from the other arguments on random walks and mapping class groups. 

As above, let  $T=(B^3,\tau)$ be a trivial tangle with $n$-strands with boundary $\partial T = \partial B^3 - \tau$, which is regarded as $S_{0,2n}$, a sphere with $2n$-punctures. 
In the curve complex $\mathcal{C}(\partial T)$, the full subcomplex spanned by the vertices represented by simple closed curves bounding a disk in $B^3 - \tau$ is called \textit{the disk set} and denoted by $D(T)$. 
See Section \ref{sec:bridgedistance} for precise definition. 

Then, to apply \cite[Theorem 5.5]{Maher:2010} with $D(T)$ as $X$, we will show the following in Section \ref{sec:Quasi-convexity}. 

\medskip
\noindent
\textbf{Lemma \ref{lemma:quasiconvex}}. 
The disk set $D(T)$ is quasi-convex in  the curve complex $\mathcal{C}(\partial T)$. 
\bigskip

We here recall why the disk set of the curve complex is applicable as the set $X$ in \cite[Theorem 5.5]{Maher:2010}. 
Actually, since by \cite[Lemma 3.2]{MasurMinsky:2004}, the relative space $\hat{G}$ is quasi-isometric to the curve complex. 
Thus we can define the disk set in $\hat{G}$ as the image under the quasi-isometry from $\mathcal{C}(\partial T)$ to $\hat{G}$. 
This disk set in $\hat{G}$, which we denote by $D(T)$ by abusing the notation, is also quasi-convex in $\hat{G}$ by the above lemma. 

\medskip

Now, the following is a key theorem in this paper, which is proved by considering the Hilden groups, and will be shown in Section \ref{sec:Hildengroup}. 

\medskip
\noindent
\textbf{Theorem 4.1}. 
Let $\mu$ be a finitely supported probability distribution on the mapping class group $\mathfrak{M}_{0,2n}$ of a $2n$-punctured sphere, whose support generates a semi-group containing a complete subgroup. 
Then the limit set of the disk set $D(T)$ in the boundary of the relative space $\hat{G}$ has measure zero with respect to both harmonic measure and reflected harmonic measure determined by random walk on $\mathfrak{M}_{0,2n}$ generated by $\mu$. 
\medskip

By virtue of these, we can apply \cite[Theorem 5.5]{Maher:2010} to obtain that 
$\hat{d} ( D(T) , w_n D(T) )$ grows linearly as $n \to \infty $ almost surely. 

Now let $T=(B_+,\tau_+)$ be the trivial tangle, and $T'=(B_-,\tau_-)$ be the mirror image of $T$. 
Then any $\varphi \in \mathcal{M}_{0,2n}$, $T \cup_\varphi  T' = (B_+,\tau_+) \cup_\varphi (B_-,\tau_-)$ gives a pair $(S^3,L)$ for a link $L = \overline{\varphi} \subset S^3$. 

Since the relative space $\hat{G}$ is quasi-isometric to the curve complex, $\hat{d}$ on $\hat{G}$ and $d$ on $\mathcal{C}(\partial T)$ are comparable. 
Actually $\hat{d} ( D(T) , w_n D(T) )$ is comparable to the \textit{Hempel distance} of the bridge splitting for the link $L$ with gluing map corresponding to $w_n$. 
See Section \ref{sec:bridgedistance} for precise definitions. 
Therefore the following completes our proof of Theorem \ref{thm: randomlinkhyp}, which will be shown in Section \ref{sec:bridgedistance}. 

\medskip
\noindent
\textbf{Theorem \ref{thm: largedistancehyp}}. 
If the Hempel distance $d(L,F)$ of a link $L$ in $S^3$ with respect to a bridge sphere $F$ is at least 3, then $L$ is a hyperbolic link. 
\medskip

In the rest of paper, we will give proofs of the above three results.

\section{Quasi-convexity of $D(T)$ in $\partial T$ }\label{sec:Quasi-convexity}

In this section, we show that the disk set $D(T)$ is quasi-convex in  the curve complex $\mathcal{C}(\partial T)$ for the trivial tangle $T$. 

Actually we prove the following, which is shown similarly as \cite[Theorem 1.1]{MasurMinsky:2004}. 

\begin{lemma}  \label{lemma:quasiconvex} 
The disk set $D(T)$ is quasi-convex in  the curve complex $\mathcal{C}(\partial T)$.
\end{lemma} 

\begin{proof}
In \cite[Theorem 1.1]{MasurMinsky:2004}, 
Masur and Minsky showed $D(H)$ s quasi-convex in  $ \mathcal{C}(\partial H)$ for a  handlebody $H$. 
In their arguments, the results used to prove \cite[Theorem 1.1]{MasurMinsky:2004} can be valid for a punctured surface except for \cite[Proposition 2.1]{MasurMinsky:2004}, which is stated as follows: 

On a surface $S$ of genus $g$ with $n$ punctures, assuming $3g - 3 + n > 1$, 
given any two essential curves $a$ and $b$ in minimal position, and an interval $J_0\subset a$ containing $a \cap b$, 
there exists a nested curve replacement sequence $\{(a_i,J_i)\}$
such that
\begin{itemize}
\item  $a_i$ is not peripheral for all $i$.    
\item If $S$ is a boundary component of a compact 3-manifold $M$ and
$a$ and $b$ are boundaries of compressing discs, then the $a_i$
can be chosen to be boundaries of compressing discs. 
\item The sequence terminates with $a_n$ homotopic to $b$.
\end{itemize}

Here we omit the definition of a nested curve replacement sequence. 
See \cite[Section 2]{MasurMinsky:2004} for details. 

Now we only need to check that the assertion corresponding to the second above also holds in our situation. 

\begin{claim} \label{claim:curvereplacement} 
For any given $a$ and $b$ in $D(T)$, there is a sequences of boundary curves of disks $a_{i}$, $a_{0}=a$, $a_{k}=b$, and $a_{i+1}$ is obtained from $a_{i}$ by a curve replacement. 
\end{claim}

\begin{proof}
We just follow the argument in the proof of \cite[Proposition 2.1]{MasurMinsky:2004}. 
There a sequence of curves is constructed inductively by using a wave curve replacement. 
In their proof, the non-peripheral property for each curve is assured automatically since the surface $S$ they consider is closed. 
On the other hand, in our situation, we can see that each curve is non-peripheral because the curve is a boundary of a disk which is boundary parallel in $T$ and a puncture appears as an end point of an arc in $\tau$. 
All the other parts of their proof are applicable. 
\end{proof}

Now by Claim \ref{claim:curvereplacement}, Theorem 1.2 of \cite{MasurMinsky:2004} and the proof of Theorem 1.1 of \cite{MasurMinsky:2004},  
Lemma \ref{lemma:quasiconvex} follows.
\end{proof}

\section{Limit set of Hilden group} \label{sec:Hildengroup}
In this section, we give a proof of the following. 

\begin{thm}\label{them:41}
Let $\mu$ be a finitely supported probability distribution on the mapping class group $\mathfrak{M}_{0,2n}$ of a $2n$-punctured sphere, whose support generates a semi-group containing a complete subgroup. 
Then the limit set of the disk set $D(T)$ in the boundary of the relative space $\hat{G}$ has measure zero with respect to both harmonic measure and reflected harmonic measure determined by random walk on $\mathfrak{M}_{0,2n}$ generated by $\mu$. 
\end{thm}

To achieve this theorem, our advantage is to consider the Hilden group. 
The \textit{Hilden group} $\mathcal{T}$ is defined as the subgroup of $\mathfrak{M}_{0,2n}$ such that $\varphi \in \mathfrak{M}_{0,2n}$  lies in $\mathcal{T}$ if and only if the a representative $f$ of $\varphi$ acting on the boundary $\partial T$ of the trivial tangle $T$ extends to the whole $T$. 
We note that, from \cite[Theorem 4.6]{MccarthyPapadopoulos:1989}, and also from \cite{KinHirose:2015}, for $n \geq 3$, there are many independent pseudo-Anosov maps in $\mathcal{T}$, while $n=2$, $\mathcal{T}$ is virtually cyclic. 
Thus, in the following, we assume that $n \geq 3$. 

In fact, Theorem \ref{them:41} is obtained from the next two theorems. 
In the rest of the section, let $\mu$ be a finitely supported probability distribution on $G=\mathfrak{M}_{0,2n}$, whose support generates a semi-group containing a complete subgroup. 
Let us consider a harmonic measure on the boundary of the relative space $\hat{G}$, 
which is the space of projective measured foliations $\mathcal{PML}(S_{0,2n})$, determined by random walk on $\mathfrak{M}_{0,2n}$ generated by $\mu$. 

\begin{thm}  \label{thm:measurezeroharmonic}
The limit set of $\mathcal{T}$  has harmonic measure zero in $\mathcal{PML}(S_{0,2n})$. 
\end{thm} 

\begin{thm}  \label{thm:limitset}
The limit set of $\mathcal{T}$ is equal to the closure of $D(T)$ in $\mathcal{PML}(S_{0,2n})$, which is connected and has empty interior if $n \geq 3$. 
\end{thm} 

First we consider Theorem \ref{thm:limitset}. 
To prove the theorem, we prepare some definitions and lemmas. 
The following are inspired by \cite{Masur:1986}.

\begin{defn} For the trivial tangle $T = (B^3 ,\tau)$ with $n$ strands $\tau$, a set of pairwisely disjoint half-disks  $\{E_{1}, E_{2}, \cdots, E_{n}\}$, i.e, $\{E_{1}, E_{2}, \cdots, E_{n}\}$ is a set of pairwisely disjoint  disks in $B^{3}$ such that for each $E_{i}$, $\partial E_{i}$ is a union of two arcs $a_i$ and $b_i$, where $a_i$ is a component of $\tau$ and $b_i$ is an arc in  $\partial B^{3}$,   is called an \emph{admissible system} for $T$. 
\end{defn}

\begin{lemma}  \label{lemma:disk}  
For a simple closed curve $\alpha$ on $\partial T$, 
$\alpha \in  D(T)$ if and only if for each admissible system  $\{E_{1}, E_{2}, \cdots, E_{n}\}$ for $T$, either $\alpha$ is disjoint from  $E_{1}, E_{2}, \cdots, E_{n}$, or there is a returning arc of $\alpha$ with respect to $\{E_{1}, E_{2}, \cdots, E_{n}\}$.
\end{lemma}

\begin{proof}  
Suppose that $\alpha \in  D(T)$ and $\alpha \cap \{E_{1}, E_{2}, \cdots, E_{n}\} \neq \emptyset$, that is,  $\alpha \cap \{b_{1}, b_{2}, \cdots, b_{n}\} \neq \emptyset$.  
Consider the dual of $\{E_{1}, E_{2}, \cdots, E_{n}\}$ in $\partial B^{3}- \tau$, that is a set of simple closed curves $c_{j}$ in $ \partial B^{3}- \tau$, such that the intersection number between of $c_{j}$ and $b_{i}$ is $\delta_{ij}$, then $c_{j}$ give a set of free generators on $\pi_{1}(B^{3}- \tau)$. 
Since $\alpha \in  D(T)$, it is trivial in $\pi_{1}(B^{3}- \tau)$, as a word in $\{c_{j}\}^{n}_{j=1}$, there is a juxtaposition $c_{j}c_{j}^{-1}$, that is,  $\alpha$ has a returning arc with respect to $b{_j}$.

For the other side, if $\alpha \in \mathcal{C}(\partial T)$ such that for each  admissible system  $\{E_{1}, E_{2}, \cdots, E_{n}\}$ of $T$, either $\alpha$ is disjoint from  $E_{1}, E_{2}, \cdots, E_{n}$, or there is a returning arc of $\alpha$ with respect to $\{E_{1}, E_{2}, \cdots, E_{n}\}$. 
From an innermost  returning arc, we can perform a surgery to get a new admissible system$\{E^{*}_{1}, E^{*}_{2}, \cdots, E^{*}_{n}\}$, and $|\alpha \cap \{E^{*}_{1}, E^{*}_{2}, \cdots, E^{*}_{n}\}|  <  | \alpha \cap \{E_{1}, E_{2}, \cdots, E_{n}\}|$. 
Inductively, we have an admissible system  such that $\alpha$ is disjoint from it, which means $\alpha$ bounds a disk in $T$. \end{proof}

\begin{proof}[Proof of Theorem \ref{thm:limitset}]
Together with Lemma \ref{lemma:disk}, the proof of Theorem \ref{thm:limitset} runs like the proof of \cite[Theorem 2.1]{Masur:1986} word by word, except that $D(T)$  has empty interior. 
Here we give an outline of the proof that $D(T)$ has empty interior. 
Take neighborhood of $b_{i}$ in $\partial B^{3}$,  we denote the boundary by $c_{i}$, and then extends $\{c_{i}\} ^{n}_{i=1}$ to a pants decomposition of $b_{i}$ in $\partial B^{3} - \cup^{n}_{i=1}\partial b_{i}$. 
With this pants-decomposition, if we take a curve $\varphi$ as in the proof of \cite[Theorem 2.1]{Masur:1986}, then $D(T)$ is shown to have empty interior in  $\mathcal{PML}(S_{0,2n})$. 
\end{proof}

Next we consider Theorem \ref{thm:measurezeroharmonic}. 
To prove the theorem, we start with the following lemma. 

\begin{lemma}  \label{lemma:characterization of limit set} 
A projective measured lamination $\beta$ lies in the limit set of the Hilden group if and only if for any admissible system  $\{E_{1}, E_{2}, \cdots, E_{n}\}$ for $T$, either $\beta  \cap (E_{1}, E_{2}, \cdots, E_{n}) = \emptyset$, or there is an index $i$, such that $\beta$ has a returning arc with respect to $\partial E_{i}$. 

\end{lemma}

\begin{proof}
With the definition admissible system, the proof runs as the proof of \cite[Lemma 1.1]{Masur:1986}. 
\end{proof}

Next, from the arguments for \cite[Proposition in Page 36]{Kerckhoff:1990}, we have the following. 

\begin{lemma}  \label{lemma:traintrack1} 
There is a recurrent maximal train track $\sigma$ in $S_{0,2n}$ such that, for the trivial tangle $T$ with the boundary $S_{0,2n}$, there is an admissible system $C$ for $T$, and a train track $\sigma'$ which is obtained from $\sigma$ by spilt at most  $18(n-1)$ times, such that no  lamination carried by $\sigma'$ has a returning arc   with respect to $C$.
\end{lemma} 

\begin{proof}
With the definition of  the admissible system, it can be shown in the same way as \cite[Proposition in Page 36 ]{Kerckhoff:1990}.
\end{proof}

This lemma directly implies the following:

\begin{lemma}  \label{lemma:traintrack2} There is a recurrent maximal train track $\sigma$ in $S_{0,2n}$, such that 
for the trivial tangle $T$ with boundary $F_{0,2n}$, $\sigma$ can be split at most $18(n-1)$ times to a train track $\sigma'$, which is disjoint from the disk set of the the tangle $T$. \qed
\end{lemma} 

Now we can show the next theorem. 

\begin{thm}  \label{thm:measurezeroLebesgue}
The limit set of $\mathcal{T}$ has Lebesgue  measure zero in $\mathcal{PML}(S_{0,2n})$. 
\end{thm} 

\begin{proof} 
With Lemma \ref{lemma:characterization of limit set}, the proof of Theorem \ref{thm:measurezeroLebesgue} just runs in the same line as the proof of \cite[Theorem 3]{Kerckhoff:1990}. 

Let $\mathcal{R}(T)$ be the set of projective measured lamination $\beta $ in $\mathcal{PML}(S_{0,2n})$ such that for every admissible system  $\{E_{1}, E_{2}, \cdots, E_{n}\}$, $\beta  \cap (E_{1}, E_{2}, \cdots, E_{n}) \neq \emptyset$, and there is a possible returning arc with respect to some $ \partial E_{i}$.

\begin{claim}  \label{claim:returning} 
$\mathcal{R}(T)$ has Lebesgue measure zero in $\mathcal{PML}(S_{0,2n})$. 
\end{claim} 

\begin{proof}
Note that there is a gap in Kerckhoff's proof, which is filled by Gadre in \cite{Gadre:2012}. 
Also see \cite{LustigMoriah:2012}. 
In \cite{Gadre:2012}, Gadre proves a uniform distortion result about Rauzy sequence on complete non-classical exchanges, which implies a uniform distortion result in the case of splittings on complete recurrent train tracks with a single switch. 
See also \cite[Page 2478]{DunfieldThurston:2006} 
on this problem. 

Then we cover $\mathcal{PML}(S_{0,2n})$ by a finite collection of complete non-classical interval exchanges.  Then, with the help of Lemma \ref{lemma:traintrack2}, \cite[Theorem 1]{Kerckhoff:1990} also holds in our case,  and by \cite[Theorem 1.2]{Gadre:2012}, Claim \ref{claim:returning} follows.
\end{proof}

Now, by  Lemma \ref{lemma:characterization of limit set}, if $\beta$ lies in the limit set of the Hilden group, then either $\beta$ lies in $\mathcal{R}(T)$, or $\beta$ has zero intersection with any admissible system for $T$. 
However the subset of projective measured laminations in $\mathcal{PML}(S_{0,2n})$ with zero intersection with one admissible system has Lebesgue measure zero, Theorem \ref{thm:measurezeroLebesgue} holds. 
\end{proof}

Consequently, with Lemma \ref{lemma:traintrack2} in stead of \cite[Theorem 3.2]{Maher:2010}, the proof of Theorem \ref{thm:measurezeroharmonic} runs in the same line as the proof of \cite[Theorem 3.1]{Maher:2010},  which states the following: 
Let $(G, \mu)$ be a random walk on the mapping class group of an orientable surface of finite type, which is not a sphere with three or fewer punctures, such that the semi-group generated by $\mu$ contains a complete subgroup. Then the harmonic measure of the disk set is zero.

\section{Bridge  distance}\label{sec:bridgedistance}

In this section, we show that a link $L$ in $S^3$ with a bridge sphere $F$ 
for which the Hempel distance $d(L,F)$ is at least 3 is a hyperbolic link. 

We will basically use the Bachman-Schleimer's criterion for a link to be hyperbolic obtained in \cite{BachmanSchleimer:2005}. 
To state the result in \cite {BachmanSchleimer:2005}, we recall the following, which is a brief summery of \cite[Section 2]{BachmanSchleimer:2005}. 

Let $L$ be a link in a closed orientable 3-manifold $M$ with a Heegaard surface $F \subset M$. 
We put $L$ \textit{in bridge position with respect to $F$}, that is, $L$ meets each of the handlebodies, say $H$ and $H'$, bounded by $F$ in a collection of trivial arcs (i.e., boundary parallel arcs). 
As notation, set $M_L = M - N(L)$, where $N(L)$ denotes a regular neighborhood of $L$, $H_L = H \cap M_L$, $H'_L = H' \cap M_L$, and $F_L = F \cap M_L$. 

We first prepare the 1-complex $\mathcal{AC} (F_L )$, called the \textit{arc and curve complex} of $F_L$, constructed as follows: 
For each proper isotopy class of essential curves in $F_L$  (i.e., not bounding a disk in $F_L$ and not boundary parallel in $F_L$), there is a vertex of $\mathcal{AC}(F_L)$. 
There is an edge of $\mathcal{AC}(F_L)$ between two distinct vertices if and only if there are representatives of the corresponding isotopy classes which are disjoint. 

Next let $V$ and $V'$ denote the sets of vertices of $\mathcal{AC} (F_L )$ corresponding to curves (properly embedded loops and arcs) that \textit{bound compressions} in $H_L$ and $H'_L$, respectively.
We say that a curve $\gamma$ in $F_L$ bounds a compression in $H_L$ (resp. in $H'_L$) if it is essential in $F_L$ and $\gamma = \partial E - \partial M_L = F_L \cap E$ holds for a \textit{cut surface} $E \subset H_L$ (resp. $E \subset H'_L$). 
An embedded surface $E$ in $M_L$ is called a cut surface if it is either 
(1) a disk $E \subset M_L$ such that $E \cap \partial M_L = \emptyset$, 
(2) a bigon $E \subset M_L$ such that $E \cap \partial M_L$ is an arc, or
(3) an annulus $E \subset M_L$ with exactly one meridional boundary component on $\partial M_L$. 
Here the disk appearing in (1) is called an \textit{essential disk} in $H_L$ or a \textit{compressible disk} for $F_L$. 

Then the \textit{Bachman-Schleimer distance} $d_{BS} (L, F)$ of $L$ with respect to $F$ is defined to be the number of edges in the shortest path from $V$ to $V'$ in $\mathcal{AC} (F_L )$. 
Now we can state the following given by Bachman and Schleimer in \cite{BachmanSchleimer:2005}. 

\begin{quote}
\cite[Corollary 6.2]{BachmanSchleimer:2005}  
If $d_{BS} (L, F)$ of a link $L$ with respect to a Heegaard surface $F$ is at least 3, then $L$ is a hyperbolic link. 
\end{quote}

\begin{remark} Even through the statement of Corollary 6.2 of \cite{BachmanSchleimer:2005} just for knot, but actually, the proof also holds for links in any manifolds. 
See \cite[Corollary 1.5]{Jang:2014} also. 
\end{remark}

On the other hand, in this paper, we consider not the arc and curve complex but the curve complex. 
Under the same setting as above, 
the \textit{curve complex} $\mathcal{C}(F_L)$ of $F_L$ is defined as the full subcomplex in $\mathcal{AC}(F_L)$ spanned by the vertices represented by simple closed curves in $F_L$. 

We define the \textit{disk set} $D (H,L)$ of $H_L$ as $V \cap \mathcal{C}(F_L)$ and $D (H',L)$ of $H'_L$ in the same way. 
Then the \textit{Hempel distance} $d(L, F)$ of $L$ with respect to $F$ is defined to be the number of edges in the shortest path from $D(H,L)$ to $D(H',L)$ in $\mathcal{C} (F_L )$. 

In general, it is well-known that $d_{BS} (L, F) \le d (L, F) \le d_{BS} (L, F) + 2$. 
In particular, in the case that $L$ is a link in the 3-sphere $S^3$ and $F$ is a bridge sphere F, 
it is shown in \cite[Proposition 1.2]{Jang:2014} that $d(L,F)=d_{BS}(L,F)$ if $d_{BS}(L,F) \ge 1$ and $d(L,F)=0$ or $1$ if $d_{BS}(L,F)=0$. 

Consequently we have:

\begin{thm}\label{thm: largedistancehyp} 
If the Hempel distance $d(L,F)$ of a link $L$ in $S^3$ with respect to a bridge sphere $F$ is at least 3, then $L$ is a hyperbolic link. 
\end{thm}

\section{Some complements}

The following can be shown in the same way as \cite[Theorem 1.1]{Maher:2010} together with the result given in \cite{Maher:2011}. 

\begin{thm} \label{thm: randomlinkdistance}  
Let $w_{k}$ be a random walk in $\mathfrak{M}_{0,2n}$ for $n \ge 3$. 
Suppose that $w _{k}$ is generated by a finitely supported probability distribution $\mu$ on $\mathfrak{M}_{0,2n}$, 
whose support generates a semi-group containing a complete subgroup.  
There are two constant $a \geq b \geq 0$ 
such that  the probability k $w_{k}$  has distance lies in $[ak, bk]$ tends to $1$ as $k \to \infty$. 
\end{thm}


By using this, together with \cite[Theorem 5.1]{BachmanSchleimer:2005} and our results in this paper, we have the following. 

\begin{thm} \label{thm: randomlinksurface}  
Let $w_{k}$ be a random walk in $\mathfrak{M}_{0,2n}$ for $n \ge 3$. 
Suppose that $w _{k}$ is generated by a finitely supported probability distribution $\mu$ on $\mathfrak{M}_{0,2n}$, 
whose support generates a semi-group containing a complete subgroup.  
For any $m \in Z$,  the probability that the complement of $\overline{w _{k}}$ in $S^{3}$ contains an essential  surface, not  with Euler characteristic number at most $m$   tends to $0$ as $k \to \infty$. 
\end{thm}


\section*{Acknowledgement}

The authors would like to thank Joseph Maher for useful discussions, in particular, about his results.

\bibliography{IchiharaMa.bib}
\bibliographystyle{amsplain}

\end{document}